\documentclass[11pt]{amsart}

\usepackage{bm}
\usepackage{fullpage}
\usepackage{amssymb}
\usepackage{amsmath,amsfonts,amsthm}
\usepackage{hyperref}
\usepackage{color}
\usepackage{cite}

\newtheorem{alphtheorem}{Theorem}

\newtheorem{theorem}{Theorem}

\newtheorem{lemma}{Lemma}

\newtheorem{definition}{Definition}
\newtheorem{claim}{Claim}
\newtheorem{proposition}{Proposition}
\newtheorem{corollary}{Corollary}

\usepackage{tikz}
\usepackage{graphicx}
\usepackage{subcaption}
\usetikzlibrary{decorations.markings}
\tikzstyle{vertex}=[circle, draw, inner sep=0pt, minimum size=6pt]

\DeclareMathOperator\Sa{\mathcal{S}}
\DeclareMathOperator\G{\mathcal{G}}

\DeclareMathOperator\I{\mathcal{I}}

\DeclareMathOperator\LL{\mathcal{L}}

\DeclareMathOperator\M{\mathcal{M}}

\DeclareMathOperator\F{\mathcal{F}}

\DeclareMathOperator\A{\mathcal{A}}
\DeclareMathOperator\B{\mathcal{B}}
\DeclareMathOperator\C{\mathcal{C}}

\def\isdef{\mbox {$\ \stackrel{\rm def}{=} \ $}}

\def\isdef{\mbox {$\ \stackrel{\rm def}{=} \ $}}

\title{  A Composition-Based Approach to EKR Problems}

\author{Javad B. Ebrahimi  and Ali Taherkhani}
\address{ J. B. Ebrahimi, 
	Department of Mathematical Sciences, Sharif University of Technology, Tehran, Iran}
\email{javad.ebrahimi@sharif.ir}
\address{A. Taherkhani, 
	Department of Mathematics, Institute for Advanced Studies in Basic Sciences (IASBS), Zanjan 45137-66731, Iran}
\email{ali.taherkhani@iasbs.ac.ir}

\begin{document}

\maketitle
\begin{abstract}

Let $\mathcal{A}$ be a family of subsets of a finite set.   
A subfamily of $\mathcal{A}$ is said to be intersecting   when any two of its members contain at least one common element.
We say that $\mathcal{A}$ is  an   Erd{\H o}s-Ko-Rado (EKR) family  if, for every element $x$ of the set, the subfamily consisting of all members of $\mathcal{A}$ that contain $x$ has the maximum cardinality among all  intersecting subfamilies of $\mathcal{A}$.
 If these subfamilies are the only maximum  intersecting subfamilies of $\mathcal{A}$, then $\mathcal{A}$ is called a strong EKR family.
In this article, we introduce a compositional framework to establish the EKR and strong EKR properties in set systems when some subfamilies are known to satisfy the EKR or strong EKR properties. 
Our method  is powerful enough to yield simpler proofs for several existing results, 
including those derived from Katona’s cycle method (1968), Borg–Meagher's admissible ordering method  (2016),  related results on  the family of permutations studied by Frankl–Deza (1977) and the family of perfect matchings of complete graphs of even order investigated by  Meagher–Moura (2005).
To demonstrate the applicability and effectiveness of our method when other existing methods have not been successful, we show that for every fixed $r$-uniform hypergraph $H$ and all sufficiently large integers $n$, the family of all subhypergraphs of the complete $r$-uniform hypergraph on $n$ vertices that are isomorphic to $H$ satisfies the strong EKR property, where two copies of $H$ are considered intersecting if they share at least one common hyperedge.
Moreover, when the structural constraint $H$ is restricted to be a cycle, we establish a series of EKR results for families of cycles in the complete graph $K_n$ and  the complete bipartite graph $K_{n,n}$ for a broad range of the parameter $n$. 
\end{abstract}

\section{Introduction}

\noindent In the sequel, $|X|$ stands for the size of the finite set $X$ and $[n]$ is a shorthand for the set $\{1,2,\ldots,n\}$.
To fix  notations, let us start with the following definition concerning {\it intersecting families} of sets and some of their most important properties.

\begin{definition}[Intersecting family, EKR and strong EKR properties]\label{def:EKRfamily}	
		Given a finite set $X$,
	\begin{itemize}
		\item [i)] Let $\mathcal{B}$ be a family of the subsets of $X$. 
		The family $\mathcal{B}$ is called an \textit{intersecting family} {\rm (}\textit{resp. $t$-intersecting family}{\rm )} if for every pair $A, B\in \mathcal{B}$ we have $A\cap 
	B\neq \varnothing$  {\rm(}resp. $|A\cap B|\geq t${\rm)}.  When a subfamily $\mathcal{B}$ of a family $\mathcal{A}$ is intersecting, we call 
	$\mathcal{B}$ an \textit{intersecting subfamily} of $\mathcal{A}$.
		\item[ii)] 
		A   family $\mathcal{A}$ of subsets of $X$ has the \textit{EKR property} if for any $x\in X$ the subfamily $\mathcal{A}_x$ of all elements of $\A$ containing $x$, has the maximum size among 
		all intersecting subfamilies of $\mathcal{A}$. When a family $\mathcal{A}$  has the EKR property, it is called an \textit{EKR family}.
		\item[iii)] If every  intersecting subfamily of $\mathcal{A}$ having the maximum size  is equal to $\mathcal{A}_x$ for some $x\in X$, we say that $\mathcal{A}$ has the \textit{strong EKR property}. When a family $\mathcal{A}$  has the strong EKR property, it is called a\textit{ strong EKR family}.
	\end{itemize}
\end{definition}
Clearly, for  an EKR family  $\mathcal{A}$ and any two elements $\{x,y\} \in X$ we have $|\mathcal{A}_x|=|\mathcal{A}_y|$. 

Our basic motivation to define the EKR property is the fundamental result of Erd{\H o}s-Ko-Rado that, strictly speaking, states that  if $n\geq 2k$ then the family of all $k$-subsets of an $n$-element set $X$, denoted by ${X\choose k}$,  has the  EKR property. Furthermore,    ${X\choose k}$ has the strong EKR property whenever $n> 2k$.
This theorem, as one of the cornerstones of extremal combinatorics was proved in 1938 and published in 1961 by Erd{\H o}s, Ko, and Rado  \cite{EKR61}. Let us recall this theorem and its generalization by Wilson \cite{wilson} (also see \cite{zbMATH03621717}) as follows.

\begin{alphtheorem}\label{thm:EKR}
	Let $n,k$, and $t$ be three  positive integers, $X$ be an $n$-element set and $\mathcal{A}$  be a $t$-intersecting subfamily  of $k$-subsets of $X$.  
	\begin{itemize}
		\item [{\rm i)}] {\rm(Erd{\H o}s-Ko-Rado Theorem \cite{EKR61})}  If $t=1$ and $n\geq 2k$, then 
		$|\mathcal{A}|\leq \binom{n-1}{k-1}$. Furthermore, when $n> 2k$, equality holds if and only if  $\mathcal{A}$ consists of all the $k$-subsets of $X$ containing a fixed element of $X$. 
       \item [{\rm ii)}]{\rm (Wilson Theorem \cite{wilson})} If $n\geq (t+1)(k-t+1)$, 	then  $|\mathcal{A}|\leq \binom{n-t}{k-t}$. Furthermore, 
    	when $n> (t+1)(k-t+1)$, equality holds if and only if $\mathcal{A}$ consists of all $k$-subsets of $X$ that contain
    	a fixed $t$-subset of  $X$. 
	\end{itemize}

\end{alphtheorem}

It is worth mentioning that  Wilson’s result was later extended by Ahlswede and Khachatrian for the case $n<  (t+1)(k-t+1)$ \cite{AhKh}.
Several generalizations and interesting alternative proofs of the  Erd{\H o}s-Ko-Rado theorem
have already appeared within the literature (e.g. see \cite{DeFr77,CaKu,LaMa,GoMeperm,BoMe1,GoMematch,DeFr1983, MR2784326, AhKh, FrWi, MR4066717,Talbot, MR0480051, MR1415313, FRANKL20121388, HilMil67, KATONA1972183, MR2285800, MR2202076, wilson, HK, BoLe, MR2156694}) while one of the approaches for generalization has been to establish similar results for other mathematical objects (families), such as the family of permutations \cite{DeFr77,CaKu,GoMeperm,LaMa,MR2202076, MR2784326}, vector spaces \cite{FrWi, HSIEH19751}, matchings in graphs \cite{GoMematch,MR4066717,BoMe1} or signed sets \cite{BoLe, MR2285800}.

Let us go through some related results in this regard before we mention our general framework.
 Two permutations $\sigma$ and $\pi$ on $[n]$  are said to be  intersecting if there exists an element $i  \in [n]$ such that $\sigma(i)=\pi(i)$.
An intersecting family of permutations is a subset of permutations such that every pair of its elements intersect.
In 1977, Frankl and Deza   showed that the maximum size of an intersecting family of permutations on $[n]$  is  $(n-1)!$
\cite{DeFr77}. They  conjectured that every maximum intersecting set is a coset of a point stabilizer, i.e. the set of all permutations sending a specific element 
to a fixed element. This conjecture is proved by Cameron and Ku \cite{CaKu} and independently by 
Larose and Malvenuto \cite{LaMa}. Later, Godsil and Meagher provided an alternative proof in \cite{GoMeperm}.

It is worth noting that  a permutation on $[n]$ corresponds  to a perfect matching in the complete bipartite graph $K_{n,n}$.
We say two matchings in a  graph  are intersecting if they have at least one common edge.
 Consequently, every intersecting family of permutations corresponds to 
an intersecting family of perfect matchings in the complete bipartite graph $K_{n,n}$. The Deza and Frankl result show that the family of all perfect matchings in
$K_{n,n}$ is an EKR family. The later result by Cameron and Ku \cite{CaKu} and independently by 
Larose and Malvenuto \cite{LaMa} showed that this family is a strong EKR family.

It is interesting that similar results hold when one extends the ambient graph to the complete graph $K_{n}$. In particular,
Meagher and Moura proved that in the complete graph on an even number of vertices, the largest 
 intersecting family of perfect matchings is precisely the family of all perfect matchings that
share one specific edge~\cite{MR2156694}. In other words, when 
$n$ is even, the family of all perfect matchings of $K_{n}$   is a strong EKR family.
Later, Godsil and Meagher presented an algebraic proof of this result \cite{GoMematch}, while Kamat and Misra extended this result to the family of $k$-matchings. 
Namely,  they showed that for even $n$ and $k < \tfrac{n}{2}$, the family of   all copies of  $k$-matchings in the complete graph $K_n$ is a strong EKR family~\cite{Kamat}. After that,   
Borg and Meagher showed that the same statement is also true when $n$ is odd and $k<\lfloor\tfrac{n}{2}\rfloor$ (see Theorems~13 and 14 in~\cite{BoMe1}).
Note that when  $k=\lfloor\tfrac{n}{2}\rfloor$,  Borg and Meagher's approach in \cite{BoMe1} does not yield the same result.

It is quite clear from this graph theoretic point of view that one may ask about the correctness of similar results when the ambient graph is replaced by a complete hypergraph or
when the structural condition is modified from being a matching to being an isomorphic copy
 of a fixed given graph $H$. In this regard, one of the first obstacles for generalization seems to be the fact that the most powerful existing methods, as Katona's cycle method or its variants (e.g. see \cite{BoMe1}), are not necessarily applicable within this general setting, primarily because a key assumption in these methods is the existence of \emph{admissible orderings}, i.e. a certain cyclic arrangement of the objects such that every consecutive $k$-interval of that ordering forms an element of the family (for a more detailed discussion see Subsection~\ref{sec:Katona}).

In what follows, we identify each subgraph/subhypergraph with its edge/hyperedge set as the ground set, i.e. all intersections are taken in the ground set, unless otherwise specified.
For instance, within this setting and by definition, a family $\mathcal{A}$ of  subgraphs of the  complete graph $K_n$ is called an \textit{EKR family},  if for every edge $e$ in $K_n$, the subfamily of  $\mathcal{A}$ consisting of all  elements that contain $e$ attains the maximum possible size among all intersecting subfamilies of $\mathcal{A}$. 
This terminology extends naturally to families of subhypergraphs of the complete $r$-uniform hypergraph $K_n^{(r)}$. 

For the ambient graph $K_n$, Theorem~\ref{thm:cycle} shows that  the family of all copies of the $k$-cycle $C_k$, satisfies the strong EKR property when $n$ is linearly larger than $k$.

\begin{theorem}\label{thm:cycle}
Let $n$ and $k$ be two positive integers.  Let $\mathcal{F}_n(C_k)$ denote the family of all $k$-cycles in $K_{n}$.
\begin{itemize}
\item[(i)] 
For any $n \geq 6$, the family $\mathcal{F}_n( C_3)$ is an EKR family, and for any $n \geq 7$, it is a strong EKR family.
\item[(ii)] For any $n\geq 24$,  
  the family  $\mathcal{F}_n( C_4)$ is an EKR family, and for any $n\geq 24$, it is a strong EKR family.
\item[(iii)] Let $k\geq 5$. 
For any  $n\geq 3k-3$,   $\mathcal{F}_n( C_k)$ is an EKR family, and for  any  $n\geq 3k-2$, it is a strong EKR family.
\end{itemize}
\end{theorem}

It is interesting that a similar result is also true when the ambient graph is the compete bipartite graph $K_{n,n}$.

\begin{theorem}\label{thm:bipcycle}
Let $n$ and $k\geq 2$ be positive integers.  Let $\mathcal{B}_n( C_{2k})$ denote the family of all $2k$-cycles in $K_{n,n}$.
For any   $n\geq 2k$, the family  $\mathcal{B}_n( C_{2k})$ is an EKR family, and for any   $n> 2k$, it is a strong EKR family. 
\end{theorem}


To show that our method is also quite capable of handling general structural constaints as well,  we will show that Theorem~\ref{thm:bipcycle} can be extended to the very general case when the structural constraint is an arbitrary connected bipartite graph $H$, by just paying the penalty of choosing $n$ sufficiently large, as follows.
  
\begin{theorem}\label{thm:bipgraphEKR}
Let $H$ be a connected bipartite graph. Then, there exists a constant $n_0(H)$ such that for every $n\geq n_0(H)$, 
the family $\B_n(H)$ consisting of all copies of $H$ in $K_{n,n}$ is a strong EKR family. 
\end{theorem}


Following the same line of thought, we prove the following theorem, that extends Theorem~\ref{thm:cycle} to a substantially more general setting of considering the family of all copies of an $r$-uniform hypergraph $H$ in the complete $r$-uniform hypergraph $K_n^{(r)}$.

\begin{theorem}\label{thm:hypergraphEKR}
Let $H$ be an $r$-uniform hypergraph. Then, there exists a constant $n_0(H)$ such that for every $n\geq n_0(H)$, 
the family  $\F_n(H)$ consisting of all  copies of $H$ in $K_{n}^{(r)}$ is a  strong EKR family.
\end{theorem}



As mentioned before, proving these results using the available means and methods existing within the literature does not seem to be a straight forward task, while we believe that our proposed composition method, discussed in detail in Section~\ref{sec:composition}, is not only strong enough to provide the above mentioned results but also it has the potential to be used in a various of different ways in different settings to obtain results about EKR and strong EKR families of some other structured families of sets.

In this regard, strictly speaking, the composition method constructs new EKR families from simpler ones in the sense that, given $\mathcal{L}$ as a family of $k$-subsets of a ground set $X$ along with an  EKR family $\mathcal{M}$ of $m$-subsets of $X$ with $k\leq m$, if for each $M \in \mathcal{M}$ there exists an EKR family $\mathcal{L}_M\subseteq \LL$ defined over the ground set $M$, then subject to some uniformity conditions, one may merge the EKR families $\mathcal{L}_M\subseteq \LL$ into $\mathcal{L}$ while preserving the EKR property. Intuitively, the method guaranties that if a family is uniformly covered through a larger EKR family while each part is EKR, then the whole family is EKR as well. The next section is dedicated to the details of such a construction.

\section{Composition Framework}\label{sec:composition}

In this section, we develop a composition framework to derive EKR-type results. This section consists of two subsections, each of which describes a lemma, namely,
the composition lemma and the $G$-balanced lemma, that will be used to establish EKR-type results in the subsequent sections.

Each subsection begins with essential definitions, followed by the formal statement of the lemma where the proofs of these lemmas are postponed to Section~\ref{sec:Defer}.
After introducing each lemma, we present illustrative examples demonstrating its applications, while in particular,   we provide new proofs of some previously known results, a couple of which are notably shorter and easier to follow.

\subsection{The Composition Lemma}
\hfill
\vspace{0.2cm}

Composing mathematical structures in a suitable way is a well-established approach in mathematics, often employed to construct new objects that satisfy desired properties. In what follows, we are going to  present a method to compose EKR families to form a new EKR family. Let us begin with a definition.

\begin{definition}
Let $\ell\leq m\leq n$ be positive integers. 
Let $\mathcal{L}$ and $\mathcal{M}$ be families of $\ell$-subsets and $m$-subsets of an $n$-element set $X$, respectively. 
\begin{itemize}
\item[i)]
A relation $\sim$ from $\mathcal{L}$ to $\mathcal{M}$ is said to be  regular, if for any $L \in \mathcal{L}$ and $M \in \mathcal{M}$, the condition $L \sim M$ implies $L \subseteq M$.
Also, if $\mathcal{I}$ is a finite set of indices and for every $i\in \mathcal{I}$  the relation   $\sim_i$ is a regular relation from $\mathcal{L}$ to $\mathcal{M}$, then we say
 $\sim_{_{\mathcal{I}}} \isdef\{\sim_i | i\in \mathcal{I}\}$ is a  family of  regular relations, from $\mathcal{L}$ to $\mathcal{M}$.
\item[ii)]
 Let  $\sim_{\I}$ be  a family of regular relations from $\mathcal{L}$ to $\mathcal{M}$. 
For  every $i\in \I$, $L\in \mathcal{L}$ and $M\in \mathcal{M}$, we define 
\[
\mathcal{M}^{^{(i)}}_{_L}\isdef\{ M\in\mathcal{M}|\,\, L\sim_i M\}.
\]
and 

\[
\mathcal{L}^{^{(i)}}_{_M}\isdef\{ L\in \mathcal{L}|\,\,  L\sim_i M\}.
\]
\end{itemize}
\end{definition}

The next definition introduces the central concept of the compositional framework.

\begin{definition}[EKR chain and special EKR chain]\label{def:EKRchain}
 Let $\ell\leq m\leq n$ be positive integers. 
Let $\mathcal{L}$ and $\mathcal{M}$ be  families of $\ell$-subsets and $m$-subsets of an  $n$-element set $X$, respectively. Assume that $\sim_{\I}$ 
 is a family of  regular relations  from $\LL$ to $\M$.
\begin{itemize}
 \item[(1)]
A triple $(\mathcal{L},\mathcal{M},\sim_{\I})$  is called an \textit{EKR chain}  if the following conditions are satisfied:
\begin{itemize}
\item[(i)] The family $\mathcal{M}$ is an EKR family. 
\item [(ii)] For every $M\in \mathcal{M}$ and $i\in \mathcal{I}$, the family $\mathcal{L}^{^{(i)}}_{_M}$ is an EKR family.
\item[(iii)] For every $M,M'\in \mathcal{M}$ and $i,j\in \I$, we have $|\mathcal{L}^{^{(i)}}_{_M}|=|\mathcal{L}^{^{(j)}}_{_{M'}}|>0$.
\item[(iv)] For every $L,L'\in \mathcal{L}$, we have $\sum\limits_{i\in\I}|\mathcal{M}^{^{(i)}}_{_L}|=\sum\limits_{i\in\I}|\mathcal{M}^{^{(i)}}_{_{L'}}|$.
\end{itemize}
\item[(2)] Let  $(\mathcal{L},\mathcal{M},\sim_{\I})$ be an EKR chain. 
The triple $(\mathcal{L},\mathcal{M},\sim_{\I})$  is called a \textit{special EKR chain}  if the  following two conditions are satisfied:

\begin{itemize}
\item[(i)] The family $\mathcal{M}$ is a strong EKR family. 
\item[(ii)] For every $M\in \mathcal{M}$ and for every $x\in M$, there exists an  $M'\in \mathcal{M}$ such that
$M\cap M'=\{x\}$.
\end{itemize}
\end{itemize}

\end{definition}

The following lemma is a key technical tool for proving the main results of this paper and is of independent interest in its own right.
\begin{lemma}[Composition Lemma]\label{lem:composition}
 Let $\ell\leq m\leq n$ be positive integers. Consider an $n$-element set $X$.
 Let $\mathcal{L}$ and $\mathcal{M}$ be  families of $\ell$-subsets and $m$-subsets of $X$, respectively, and
  let $\sim_{\I}$ be a  family  of regular relations from $\LL$ to $\M$.
\begin{itemize}
\item[(i)] If $(\mathcal{L},\mathcal{M},\sim_{\I})$ is an EKR chain, then $\mathcal{L}$ is an EKR family. 
\item[(ii)] If  $(\mathcal{L},\mathcal{M},\sim_{\I})$ is a special EKR chain, then  $\mathcal{L}$ is a strong EKR family.
\end{itemize}
\end{lemma}

  The proof of Lemma~\ref{lem:composition} is deferred to Section \ref{sec:Defer} in order to avoid interrupting the continuity of discussion, where we first
 concentrate on applications that highlight usefulness of the composition lemma.
 
Before we proceed to some new results, let us provide a proof of the classical Erd{\H o}s-Ko-Rado Theorem using the composition setup that essentially shows that the composition framework already covers  the elegant argument of Katona, known as Katona's cycle method.

\begin{corollary}[Erd{\H o}s-Ko-Rado Theorem \cite{EKR61}]
    Let $n\geq 2k$ be two positive integer numbers. Then, the family of all $k$-subsets of an n-element set $X$ is an EKR family.
\end{corollary}

\begin{proof}
Let $X=\{0,1,\ldots,n-1\}$, $\LL={X\choose k}$, $\M =\{X\}$, and $\I = {S}_n$, where ${S}_n$  denotes the symmetric group of all permutations on $X$. 
For  $\pi \in {S}_n$ and an element $i\in X$, we define an  interval  $A_{\pi,i}$ by
  $$A_{\pi,i}\isdef\{\pi(i),\pi(i+1),\ldots,\pi(i+k-1)\},$$
where addition is modulo $n$.

For every element $\pi \in {S}_n$, define the relation $\sim_{\pi}$ as follows.  A $k$-subset $L$ of $X$  satisfies $L\sim_{\pi} X$ if  for some $i\in X$
 we have $L=A_{\pi,i}$.
 
 We show that the triple $(\LL,\M,\sim_{{S}_n})$ is an EKR chain. To see this, first note that for every fixed $\pi \in {S}_n$
 and $M=X$,  the family $\LL^{^{(\pi)}}_{_M}=\{L|\, L\sim_{\pi} M\}=\{A_{\pi,i}|\, i\in X \} $ is an EKR family. 
 This follows since,  when $n\geq 2k$, the maximum cardinality of an intersecting family among all
  $k$-subsets of $X$ that are an interval with respect to $\pi$ is equal to $k$, i.e., the number of  intervals  with respect to
$\pi$ that contain a fixed element of $X$.  
 Finally, the family $\M$ is an EKR family  since it consists of a single set, and obviously satisfies the EKR property. 
  Clearly,  for any two elements $\{\pi, \sigma \} \subseteq {S}_n$, we have $|\LL^{^{(\pi)}}_{_M}|=|\LL^{^{(\sigma)}}_{_M}|=n$.
 Also, for any $k$-subset $L$ and any $\pi \in {S}_n$, if there exists an $i \in X$ such that $L = A_{\pi,i}$, then $|\mathcal{M}^{^{(\pi)}}_L| = 1$; otherwise, 
 $|\mathcal{M}^{^{(\pi)}}_L| = 0$. Therefore, 
 we have $\sum\limits_{\pi\in {{S}_n}}|\mathcal{M}^{^{(\pi)}}_{_L}|= n\times k!\times(n-k)!$.
Now by  Part~(i) of Lemma \ref{lem:composition}, it follows that the family $\binom{X}{k}$ is an EKR family.
\end{proof}
 In what follows, we will also
present several more examples that are not merely translations of known methods into our composition framework.

\subsection{$G$-balanced Lemma} 
\hfill
\vspace{0.2cm}

Let $X$ be a finite set of size $n$, and let $G$ be a finite group that acts transitively on $X$, 
     i.e. for every $x,x'\in X$, there exists $g\in G$ such that $x'=gx$.
The action of $G$ on $X$ naturally extends to an action of $G$ on $\binom{X}{k}$ as follows. 
For every $ A\in \binom{X}{k}$ and $g\in G$, define $gA\isdef\{ga: a\in A\}$. 
\begin{definition}[$(G,j)$-\text{balanced family}]\label{balance}
Let $\mathcal{F}$ be a subfamily of $\binom{X}{k}$, the group $G$ be a finite group acting transitively on $X$, and $j$ be a positive integer. We say that $\mathcal{F}$ is $(G,j)$-\textit{balanced} if the following conditions are satisfied,
\begin{itemize}
    \item[i)] The group $G$ acts transitively on $\mathcal{F}$.
    \item[ii)] There exist  elements $D_1,D_2,\ldots,D_r$ in $\mathcal{F}$ such that every element of $X$ belongs to precisely $j$ of the sets $D_i$. 
    \item[iii)] No more than $j$ of the  sets $D_i$  form an intersecting family.
\end{itemize}

\end{definition}

Observe that in the above definition, when $j=1$, the second condition implies the third one.

The following lemma is a consequence of Lemma~\ref{lem:composition} and is useful in proving that certain families of objects are EKR families. 

\begin{lemma}[$G$-balanced Lemma]\label{lem:gbalanceisEKR}
If $G$ acts transitively on a set $X$  and $\mathcal{F}\subseteq \binom{X}{k}$ is $(G,j)$-balanced, then $\mathcal{F}$ is an EKR family.
\end{lemma}
 To maintain the continuity of the discussion, the proof of Lemma~\ref{lem:gbalanceisEKR} is deferred to Section \ref{sec:Defer}, while the following corollary  presents the Frankl and Deza theorem for permutations.

\begin{corollary}{\rm \cite{DeFr77}}
    Let $ n$ be a positive integer number and let $S_n$  denote the symmetric group of all permutations on $[n]$.  Then, $S_n$ is an EKR family; that is,
$S_{i,j}\isdef\{\sigma\in S_n|\, \sigma(i)=j\}$ is a maximum intersecting family of permutations for every $i, j\in [n]$.
\end{corollary}

\begin{proof}
Note that a perfect matching in the complete bipartite graph $K_{n,n}$ corresponds to a unique permutation on $[n]$. 
Consequently, every intersecting family of permutations can be identified with an intersecting family of perfect matchings in the complete bipartite graph $K_{n,n}$.

Let $\mathcal{F}$ denote the family of all perfect matchings in $K_{n,n}$. The edge set of $K_{n,n}$ can be decomposed into perfect matchings. Furthermore, the group $S_n\times S_n$ acts on the edge set of $K_{n,n}$ by $(\sigma,\sigma')(i,j)\isdef (\sigma(i),\sigma'(j))$ for all $\sigma,\sigma'\in S_n$ and $i,j\in [n]$. This action and its extension to the set of all perfect matchings of $K_{n,n}$ is trivially transitive.  Therefore, $\mathcal{F}$ forms an $(S_n\times S_n,1)$-balanced family. By Lemma~\ref{lem:gbalanceisEKR}, it follows that $\mathcal{F}$ satisfies the EKR property.
\end{proof}

 For some more applications of G-balanced lemma see Theorem~\ref{prop:borg} and Proposition~\ref{prop:Hamiltonian}.

\section{Applications}

In this section, we present some   applications of the composition  and G-balanced lemmas developed in the previous section. Each subsection of this part is devoted to one such application.

\subsection{Generalized Katona Cycle Method}\label{sec:Katona}
\hfill
\vspace{0.2cm}

Borg and Meagher in \cite{BoMe1} presented an elegant framework  to derive  EKR-type results for set systems under certain symmetry assumptions. Their result is a natural generalization of  celebrated Katona's cycle method for proving the original Erd{\H o}s-Ko-Rado Theorem. The following result is in the heart of their framework, but in a slightly different notation.

\begin{alphtheorem}[\cite{BoMe1}, Theorem 7]\label{prop:borg}
Assume that there is a family $\mathcal{V}$ of $k$-subsets of a set $X$ of size $n$, such that $n\geq 2k$. Assume that the following conditions hold.
\begin{itemize}
    \item[1)] There is a group $G$ that acts transitively on $X$ and through this action, also acts transitively on $\mathcal{V}$.
    \item[2)] There is an ordering of all the elements of $X$ around a circle such that any $k$ consecutive elements of the ordering form an element of the family $\mathcal{V}$.
\end{itemize}

Then, the family $\mathcal{V}$ has the EKR property.
\end{alphtheorem}
In what follows, we present a short proof of Theorem~\ref{prop:borg}  using  Lemma \ref{lem:gbalanceisEKR}. 

\begin{proof}[Proof of  Theorem~\ref{prop:borg}] 
We show that the assumptions of  Theorem~\ref{prop:borg} imply those of Lemma \ref{lem:gbalanceisEKR}. 
The first condition of Definition~ \ref{balance} regarding the existence of the group $G$ is precisely the first assumption of Theorem~\ref{prop:borg}. The 
second and third conditions of Definition~ \ref{balance}
  are satisfied by Condition~(2) of  Theorem~\ref{prop:borg}, simply by taking the sets $D_i$ to be all the consecutive $k$-elements in the ordering specified by Condition~(2).
It follows that each element  of $X$ is contained in exactly $k$ of the sets $D_i$ and  at most $k$ of the  sets $D_i$  form an intersecting family.
Therefore,   Lemma \ref{lem:gbalanceisEKR} guarantees that the family $\mathcal{V}$ has the EKR property.
\end{proof}

An ordering of the elements of $X$ which satisfies Condition $2$ in  Theorem~\ref{prop:borg} is called an \textit{admissible ordering}.
 In   \cite{BoMe1},  Borg and Meagher   showed that if $X$ is the set of all edges of 
 the complete graph $K_{n}$, and $\mathcal{V}$ is the set of all $k$-matchings of $K_{n}$ where
  $k< \lfloor{n\over 2}\rfloor$, then there exists an admissible ordering of $X$. In other words, there exists an ordering of all the edges of $K_{n}$ around a circle such that any $k$ 
  consecutive edges on the circle form a $ k$-matching. Then, as an interesting consequence of Theorem~\ref{prop:borg}, they concluded that for $k< \lfloor\tfrac{n}{2}\rfloor$, 
  the family of all $k$-matchings in $K_{n}$ has the EKR property. This result is an extension of an earlier result of  Kamat and  Misra for even $n$ in~\cite{Kamat}.

Although Theorem~\ref{prop:borg} is a powerful tool to prove EKR-type results, it has its limitations. 
Most notably, the existence of admissible orderings is a strong assumption that  restricts its applicability. For instance, when $n\geq 4$, if we consider the subsets of the edges of $K_n$ that form a cycle $C_k$ (instead of $k$-matchings), then there is no such admissible ordering. 
Equivalently, there is no way to order all the edges of $K_n$ around a circle such that any $k$ consecutive edges form a $k$-cycle. 
Therefore, their method does not resolve the question of 
whether the family of all $k$-cycles  of $K_n$ has the EKR property. 
In the next subsection, we show that the composition lemma can handle  this question.

\subsection{The EKR Property for Cycles and Matchings}\label{subsec:matchcycle}
\hfill
\vspace{0.2cm}

As the next application of the composition lemma, we prove that the  family of all $k$-cycles  of $K_n$ and $K_{n,n}$, for sufficiently large $n$, has the EKR property.  
Also, we show that the family of all $k$-matchings of $K_n$ and $K_{n,n}$, posses  the EKR property.
 	The first step of the proof is the following proposition, which we prove using Lemma \ref{lem:gbalanceisEKR}.

\begin{proposition} \label{prop:Hamiltonian}~
\begin{itemize} 
\item[(i)] Let $n\geq 5$  be an integer and $\F_n(C_n)$ be the family of  all Hamiltonian cycles in the complete graph $K_n$.
Then, the family $\F_n(C_n)$ is an EKR family.
\item[(ii)] Let $n\geq 4$  be an integer and  let $\mathcal{B}_n(C_{2n})$ be the family of all Hamiltonian cycles in the complete bipartite graph $K_{n,n}$.
Then the family $\mathcal{B}_n(C_{2n})$ is an EKR family.
\end{itemize} 
\end{proposition}
\begin{proof}
	
To prove (i),
first, let  $n\geq 5$ be an odd integer. By the Walecki construction (e.g.  see \cite{alspach, Bryant2007}), $K_n$ has an edge decomposition to Hamiltonian cycles, say 
$\{C_1,C_2,\ldots,C_{n-1\over 2}\}$. 
For even $n$, it is well-known that by using the circle method, $K_n$ can be decomposed into perfect matchings $\{N_1,N_2,\ldots, N_{n-1}\}$
(e.g. see Section~7.1 of \cite{west}).
  Now we define $n-1$ Hamiltonian cycles as follows, $C_1\isdef N_1\cup N_2 ,C_2 \isdef N_2\cup N_3,\ldots,C_{n-1} \isdef N_{n-1}\cup N_1$.
   For this construction, each edge appears in exactly two $C_i$'s, and at most two of the $C_i$'s form an intersecting family.

  The symmetric group $S_n$ (the permutation group of the vertices of $K_n$) naturally acts on both $E(K_n)$ and $\F_n(C_n)$.
Therefore, when $n$ is odd, $\F_n(C_n)$ is $(S_n,1)$-balanced, and when $n$ is even, $\F_n(C_n)$ is $(S_n,2)$-balanced. Consequently, 
 by Lemma~\ref{lem:gbalanceisEKR},  $\F_n(C_n)$ is an EKR family.

To prove (ii),
let $n\geq 4$ be an integer. Consider the complete bipartite graph $K_{n,n}$ with parts $A=\{u_0,\ldots,u_{n-1}\}$  and $B=\{v_0,\ldots,v_{n-1}\}$. We consider
the following decomposition of the edges of $K_{n,n}$ into perfect matchings.
For each $0\leq i\leq n-1$, let $N_i$ be a perfect matching in $K_{n,n}$ with the edge set $E(N_i)=\{u_{j} v_{j+i}|\, 0\leq j\leq n-1\}$ where addition is modulo $n$.
  Now we define $n$ Hamiltonian cycles as follows, $C_0\isdef N_0\cup N_1 ,C_1 \isdef N_1\cup N_2,\ldots,C_{n-1} \isdef N_{n-1}\cup N_0$.
   For this construction, each edge appears in exactly  two of the sets $C_i$, and at most  two of the  sets $C_i$ intersect.

Let $S_A$ and $S_B$ be the permutation groups of vertices of $A$ and $B$, respectively.
The symmetric group $S_A\times S_B$  naturally acts on both $E(K_{n,n})$ and $\mathcal{B}_n(C_{2n})$.
Therefore,  $\mathcal{B}_n(C_{2n})$ is $(S_A\times S_B,2)$-balanced. Consequently, by Lemma~\ref{lem:gbalanceisEKR}, $\mathcal{B}_n(C_{2n})$ is an EKR family.

 \end{proof}


\noindent We say that a $k$-subset of $[n]$ is   {\it separated}   if it does not contain any pair of consecutive elements $ i, i+1$ or the pair $n,1$.
Let  the family of all separated sets in ${[n]\choose k}$ be denoted by  ${[n]\choose k}_2$. 
 Holroyd and Johnson posed a conjecture in~\cite{HoJo, Holroyd1999} that  an analogue of the Erd{\H o}s– Ko–Rado theorem holds for intersecting families of separated $k$-sets, i.e.
 ${[n]\choose k}_2$ is an EKR family.
 Their conjecture was settled by Talbot in~\cite{Talbot}. For any $i\in [n]$, let $\Sa_i^*\isdef\{A|\, A\in  {n\choose k}_2 \text{ and }i\in A\}$.
\begin{alphtheorem}{\rm\cite{Talbot}}\label{thm:Talbot}
Let $n\geq 2k$ and $\F\subset {[n]\choose k}_2$ be an intersecting family. Then, 
$|\F|\leq |\Sa_1^*|$.
 Moreover, for $n\neq 2k+2$ the only maximum intersecting subfamilies are $S_i^*$ for $i\in [n]$. If $n=2k+2$, then other maximum intersecting subfamilies exist.
\end{alphtheorem}

Let  $k $ be a positive integer. The $k$-matching graph, denoted by $T_k$, is the union of $k$ vertex-disjoint
edges,  i.e.  the union of $k$ copies of $K_2$.
Let $n\geq 2k$  and let $\C_n(T_k)$ be the family of all $k$-matchings in  the cycle $C_n$. From Talbot's theorem (Theorem~\ref{thm:Talbot}), $\C_n(T_k)$ is an EKR family. This observation leads to the following corollary. 
\begin{corollary}
Let $n$ and $k$ be positive integers with $n\geq 2k$. 
 Let $\F_n(T_k)$ denote the family of all $k$-matchings in $K_n$. Then,  $\F_n(T_k)$ is an EKR family.
  \end{corollary}
\begin{proof}
For $n=2$ the statement is trivially true. Assume that $n\geq 3$. Let $\F_n(C_n)$ be the family of all  Hamiltonian cycles in the complete graph $K_n$.
We claim that $\big(\F_n(T_k),\F_n(C_n), \{\subseteq\}\big)$ is an EKR chain in which $\{\subseteq\}$ is the set of a single inclusion relation.
We verify the conditions  of Definition~\ref{def:EKRchain}, to prove this claim  by  Part~(i) of Lemma~\ref{lem:composition}.
The fact that $\F_n(C_n)$ is an EKR family has been proven in Part~(i) of Proposition~\ref{prop:Hamiltonian}. 
For every fixed Hamiltonian cycle $C\in \F_n(C_n)$, by applying Theorem~\ref{thm:Talbot}, 
we know that the family of all $k$-matchings  in $C$ is an EKR family. 
The fact that all Hamiltonian cycles in $K_n$ contain the same number of $k$-matchings and every $k$-matching is contained in the same number of Hamiltonian cycles in $K_n$
 is trivial due to the symmetry.
Therefore,  by Part~(i) of Lemma~\ref{lem:composition}, the family of all $k$-matchings in $K_n$  is an EKR family.
\end{proof}
\begin{corollary}
Let $n$ and $k$ be two positive integers with  $n\geq k$. 
 Let $\mathcal{B}_n(T_k)$ denote the family of all $k$-matchings in $K_{n,n}$. Then,  $\mathcal{B}_n(T_k)$ is an EKR family.   
 \end{corollary}
\begin{proof}
The case $n=1$ is obvious. Assume that
 $ n\geq 2$.
We show that $\big(\mathcal{B}_n(T_k),\mathcal{B}_n(C_{2n}), \{\subseteq\}\big)$ is an EKR chain.
We now check the conditions for being an EKR chain in Definition~\ref{def:EKRchain}, to prove this claim.
The fact that $\mathcal{B}_n(C_{2n})$ is an EKR family has been proven in Part~(ii) of Proposition~\ref{prop:Hamiltonian}. 
For every fixed Hamiltonian cycle $C\in \B_n(C_{2n})$, by applying  Theorem~\ref{thm:Talbot}, we know that the family of all $k$-matchings 
  in $C$ is an EKR family. 
The fact that all Hamiltonian cycles in $K_{n,n}$ contain the same number of $k$-matchings and every $k$-matching is contained in the same number of Hamiltonian cycles in $K_{n,n}$
 is trivial due to the symmetry.
Therefore,  by Lemma~\ref{lem:composition},
the family of all $k$-matchings in $K_{n,n}$  is an EKR family.
\end{proof}



We now present the proof of Theorem \ref{thm:cycle}.
\begin{proof}[Proof of Theorem \ref{thm:cycle}]

We divide the proof into three cases depending on $k$.
\begin{itemize}
\item Case (i): $k=3$.
In this case, two $3$-cycles share an edge if and only if their vertex sets have two common vertices. Then, the assertion follows from Part~(ii) of Theorem~\ref{thm:EKR} by taking $X=V(K_n)$, $k=3$, and $t=2$.

\item Case (ii): $k\geq 5$.
In Lemma~\ref{lem:composition}, we take $\LL=\mathcal{F}_n( C_k)$, $\mathcal{M}= \mathcal{F}_n(K_k)$, and $\sim_{\I}=\{\subseteq\}$. 
We show that $\big(\mathcal{F}_n( C_k), \mathcal{F}_n(K_k),\{\subseteq\}\big)$ is an EKR chain.
We proceed to verify each of the conditions stated in
 Definition~\ref{def:EKRchain}.
  
  Two $k$-cliques in $K_n$ share an edge if and only if their vertex sets have two common vertices. 
Since $n\geq 3(k-1)$,  by taking $X=V(K_n)$ and $t=2$, it follows from Part~(ii) of Theorem~\ref{thm:EKR}   that $\mathcal{F}_n(K_k)$ is an EKR family.
Let
 $Q \in \mathcal{F}_n(K_k)$ be a $k$-clique in the complete graph $K_n$, and define  
 $$\mathcal{L}_{Q}\isdef  \{ C \mid C \text{ is a } k\text{-cycle contained in } Q \}.$$  
 As established in Part~(i) of Proposition~\ref{prop:Hamiltonian}, the family $\mathcal{L}_{Q}$ satisfies the EKR property. 
 Note that each $k$-clique in $K_{n}$ contains exactly $\tfrac{(k-1)!}{2}$  $k$-cycles and every $k$-cycle is contained in exactly one $k$-clique  in $K_{n}$.
 Consequently, the assertion  follows by 
Part~(i) of the composition lemma.
 
 Now assume that $n> 3(k-1)$. Then,  by taking $t=2$, it follows from Part~(ii) of Theorem~\ref{thm:EKR}   that $\mathcal{F}_n(K_k)$ is a strong EKR family.
Assume that $Q_1\in \mathcal{F}_n(K_k)$ and $xy$ is one edge of $Q_1$. Since $n>3(k-1)>2k-2$, there exists $Q_2\in \mathcal{F}_n(K_k)$ such that
$E(Q_1)\cap E(Q_2)=\{xy\}.$
Therefore, $(\mathcal{F}_n( C_k), \mathcal{F}_n(K_k), {\subseteq})$ is a special EKR chain, and hence, by Part~(ii) of the composition lemma, $\mathcal{F}_n( C_k)$ is a strong EKR family.
 \item Case (iii): $k=4$.
  Since the edge set of $K_4$ does not decomposes into copies of $C_4$ but  $K_9$ admits such a decomposition (see \cite[Theorem1.1]{Bryant}),
   we work with $\mathcal{F}_n(K_9)$ instead of $\mathcal{F}_n(K_4)$. 
Since $n\geq 24$, it follows from Theorem~\ref{thm:EKR}   that $\mathcal{F}_n(K_9)$ is an EKR family.
 Take $\LL=\mathcal{F}_n( C_4)$, $\mathcal{M}= \mathcal{F}_n(K_9)$, and $\sim_{\I}=\{\subseteq\}$.
 We show that $\big(\mathcal{F}_n( C_4), \mathcal{F}_n(K_9),\{\subseteq\}\big)$ is an EKR chain and a special EKR chain, when $n\geq 24$ and $n\geq 25$, respectively.
Let
 $Q \in \mathcal{F}_n(K_9)$ be a $9$-clique in the complete graph $K_n$. 
 Define  
 
 $$\mathcal{L}_{Q}\isdef  \{ C \mid C \text{ is a } 4\text{-cycle contained in } Q \}.$$  
 
\noindent 
The rest of the proof proceeds by an argument similar to that  in the case $k \geq 5$.

 \end{itemize}
\end{proof}

\begin{proof}[Proof of Theorem \ref{thm:bipcycle}]
Consider the bipartition $(X,Y)$  of  the complete bipartite graph $K_{n,n}$. 
Let $\mathcal{B}_n(K_{k,k})$ denote the family of all subgraphs of $K_{n,n}$ that are isomorphic to $K_{k,k}$.
In the composition lemma (Lemma~\ref{lem:composition}), we take $\LL=\mathcal{B}_n( C_{2k})$, $\mathcal{M}= \mathcal{B}_n(K_{k,k})$, and $\sim_{\I}=\{\subseteq\}$, where $\subseteq$ denotes the subgraph inclusion relation. 
We show that $\big(\mathcal{B}_n( C_{2k}), \mathcal{B}_n(K_{k,k}),\{\subseteq\}\big)$ is an  EKR chain whenever $n\geq 2k$ and moreover,
 it  is a special EKR chain whenever $n>2k$. Consequently, the desired result follows by 
 Parts~(1) and (2) of  the composition lemma.

We proceed to verify each of the conditions stated in
 Definition~\ref{def:EKRchain}.
 To this end,  we require the following claim. We postpone the proof of Claim~\ref{claim}  until the end of  Section~\ref{sec:Defer}.
\begin{claim}\label{claim} 
For $n\geq 2k$,  the family $\mathcal{B}_n(K_{k,k})$ is an  EKR family,  and for any $n > 2k$, it is a strong EKR family.

\end{claim}

For any
 $Q \in \mathcal{B}_n(K_{k,k})$, define  
 $$\mathcal{L}_{Q}\isdef  \{ C|\, C \text{ is a } 2k\text{-cycle contained in } Q \}.$$
   
 As established in Part~(ii) of Proposition~\ref{prop:Hamiltonian}, the family $\mathcal{L}_{Q}$ satisfies the EKR property.
 
 Each $2k$-cycle $C$ is contained in exactly one $Q\in  \mathcal{B}_n(K_{k,k})$. 
 Also, each $Q\in  \mathcal{B}_n(K_{k,k})$ contains $\tfrac{(k-1)!k!}{2}$ $2k$-cycles. 
 Therefore,  $\big(\mathcal{B}_n( C_{2k}), \mathcal{B}_n(K_{k,k}),\{\subseteq\}\big)$ is an  EKR chain for $n\geq 2k$.

 By Claim~\ref{claim}, we know that $\mathcal{B}_n(K_{k,k})$ is a strong EKR family for $n>2k$. To verify the second condition in Part (2) of
 the composition lemma, 
assume that  $Q_1 \in \mathcal{B}_n(K_{k,k})$ is a complete bipartite graph  with bipartition $(A_1,B_1)$. 
Take a $k$-subset $A_2\subset X$ and  a $k$-subset $B_2\subset Y$ such that 
$|A_1\cap A_2|=1$ and 
$|B_1\cap B_2|=1$. Consider the complete bipartite graph $Q_2$ with bipartition $(A_2,B_2)$.   It is straight forward to verify  that $|E(Q_1)\cap E(Q_2)|=1$.
Therefore, $\big(\mathcal{B}_n( C_{2k}), \mathcal{B}_n(K_{k,k}),\{\subseteq\}\big)$ is a special   EKR chain for $n> 2k$.

\end{proof}
\subsection{ The EKR Property for $H$-Copies in the Complete Bipartite Graph $K_{n,n}$}

\hfill

To prove Theorem~\ref{thm:bipgraphEKR}, we need to show that there exists a positive integer $n$ such that 
the edge set of  $K_{n,n}$ can be decomposed into copies of a fixed bipartite graph $H$. 
This can be viewed as a bipartite analogue of Wilson’s theorem on edge decompositions of complete graphs into copies of a fixed graph 
$H$~\cite{Wilson2}.
\begin{alphtheorem}{\rm\cite{Haggkvist}\label{thm:bip}}
For any bipartite graph $H$, there exists an positive integer $n=n(H)$ such that  the edge set of $K_{n,n}$ can be decomposed into subsets each of which forms the edge set of a copy of $H$.
\end{alphtheorem}

\begin{proof}[Proof of Theorem~\ref{thm:bipgraphEKR}]
By using Theorem~\ref{thm:bip}, there exists  a positive integer $n_0=n(H)$ such that  the edge set of $K_{n_0,n_0}$ can be decomposed into subsets each of which forms the edge set of a copy of $H$.

Let $n> 2n_0$. Let $\mathcal{B}_n(K_{n_0,n_0})$ denote the family of all subgraphs of $K_{n,n}$ that are isomorphic to $K_{n_0,n_0}$.
Let $\mathcal{B}_n(H)$ denote the family of all subgraphs of $K_{n,n}$ that are isomorphic to $H$.
In Lemma~\ref{lem:composition}, take $\LL=\mathcal{B}_n(H)$, $\mathcal{M}= \mathcal{B}_n(K_{n_0,n_0})$, and $\sim_{\I}=\{\subseteq\}$, where $\subseteq$ denotes the subgraph inclusion relation. 
We show that $\big(\mathcal{B}_n(H), \mathcal{B}_n(K_{n_0,n_0}),\{\subseteq\}\big)$ is a special EKR chain.
We proceed to verify each of the conditions stated in
 Definition~\ref{def:EKRchain}.
 
\noindent  First note that by Claim~\ref{claim},  $\mathcal{B}_n(K_{n_0,n_0})$ is a strong  EKR family.

\noindent For any
 $Q \in \mathcal{B}_n(K_{n_0,n_0})$, define  
 $$\mathcal{L}_{Q}\isdef  \{ H'|\, H' \text{ is a subgraph of } Q \text{ that is isomorphic to } H \}.$$

 \noindent Now, we show that the family $\mathcal{L}_{Q}$ satisfies the EKR property.
  This claim follows from the fact that  $\mathcal{L}_{Q}$ is  $(S_{n_0}\times S_{n_0}\times {\mathbb Z}_2,1)$-balanced, where $S_{n_0}$
   is the permutation group on the vertex of one part of $Q(=K_{n_0,n_0})$.
Observe that the automorphism group of the complete bipartite graph $K_{n_0,n_0}$ is
\[
\operatorname{Aut}(K_{n_0,n_0}) \cong \left(S_{n_0} \times S_{n_0}\right) \times \mathbb{Z}_2.
\]
Here, each symmetric group $S_{n_0}$ acts on one of the two parts of the graph, while the factor $\mathbb{Z}_2$ corresponds to swapping the two parts.

Note that the action of the group $S_{n_0}\times S_{n_0}\times {\mathbb Z}_2$ on the set of the vertices of $Q(=K_{n_0,n_0})$ naturally extends to the set of the edges of $Q(=K_{n_0,n_0})$, and in turn, to the set 
of the copies of $H$ in $Q(=K_{n_0,n_0})$, i.e. $\B_{n_0}(H)$.

  Since $K_{n_0,n_0}$ can be decomposed into disjoint copies of $H$,  
Conditions~(ii) and (iii) of  Definition~\ref{balance} for being $(S_{n_0}\times S_{n_0}\times 
{\mathbb Z}_2,1)$-balanced hold.
Therefore, $\LL_Q$ is in fact $(S_{n_0}\times S_{n_0}\times {\mathbb Z}_2,1)$-balanced. Now, Lemma \ref{lem:gbalanceisEKR} guarantees that $\LL_Q$ is an EKR family. 
  Note that  by symmetry each copy of $H$ lies in the same number of copies of $K_{n_0,n_0}$ in $K_{n_,n}$ and  each copy of $K_{n_0,n_0}$ in $K_{n,n}$ 
 contains the same number of copies of $H$.
 Hence, $\big(\mathcal{B}_n(H), \mathcal{B}_n(K_{n_0,n_0}),\{\subseteq\}\big)$ is an EKR chain.
 

Consider $Q_1 \in \mathcal{B}_n(K_{n_0,n_0})$   on the  bipartition $(A_1,B_1)$. Take an $n_0$-subset $A_2\subset X$ and  an $n_0$-subset $B_2\subset Y$ such that 
$|A_1\cap A_2|=1$ and 
$|B_1\cap B_2|=1$. Consider the complete bipartite graph $Q_2$ with bipartition $(A_2,B_2)$.   It is easy to verify that $|E(Q_1)\cap E(Q_2)|=1$.

 Consequently, the desired result follows by 
 Parts (1) and (2) of  the composition lemma.
\end{proof}

\subsection{ The EKR Property for $H$-Copies in Uniform Hypergraphs}

\hfill
\vspace{0.1cm}

 
Before proving  Theorem~\ref{thm:hypergraphEKR}, we recall a useful and interesting result from \cite{MR4572074}, due to   Glock,  K{\" u}hn,  Lo and  Osthus, concerning the decomposition of the hyperedges of a complete $r$-uniform hypergraph into copies of a given $r$-uniform hypergraph.
 It is worth mentioning that in the case of graphs (i.e., when $r = 2$), this result was previously proved by Wilson in \cite{Wilson2}.

\begin{alphtheorem}[Weak version of Theorem 1.1 in \cite{MR4572074}]\label{prop:decomposition}
For any $r$-uniform hypergraph $H$, there exists an integer $n=n(H)$ such that  the hyperedge set of the complete $r$-uniform graph 
$K_n^{(r)}$ can be decomposed into subsets each of which forms the hyperedge set of a copy of $H$.
\end{alphtheorem}





We are now ready to prove Theorem \ref{thm:hypergraphEKR}.
\begin{proof}[Proof of Theorem \ref{thm:hypergraphEKR}]
Let $H$ be any arbitrary $r$-uniform hypergraph. By Theorem~\ref{prop:decomposition}, there exists a positive integer  $n_0$ such that the edge set of the complete $r$-uniform hypergraph $K_{n_0}^{(r)}$ can be decomposed into copies of $H$.
 Let $n> (r+1)(n_0-r+1)$
be an arbitrary integer number.   Let $\mathcal{F}_n(K_{n_0}^{(r)})$ denote the family of all subhypergraphs of $K_{n}^{(r)}$ that are isomorphic to $K_{n_0}^{(r)}$.
Let $\F_n(H)$ denote the family of all subhypergraphs of $K_{n}^{(r)}$ that are isomorphic to $H$.
In Lemma~\ref{lem:composition}, take $\LL=\mathcal{F}_n(H)$, $\mathcal{M}= \mathcal{F}_n(K_{n_0}^{(r)})$, and $\sim_{\I}=\{\subseteq\}$, where $\subseteq$ denotes the subhypergraph inclusion relation. 
We show that $\big(\mathcal{F}_n(H), \mathcal{F}_n(K_{n_0}^{(r)}),\{\subseteq\}\big)$ is a special EKR chain.
We proceed to verify each of the conditions stated in
 Definition~\ref{def:EKRchain}.



We first show that $ \mathcal{F}_n(K_{n_0}^{(r)})$ is a strong EKR family.  Note that two copies of  $K_{n_0}^{(r)}$ share a hyperedge if and only if their vertex sets have $r$ common vertices. 
 By taking $k=n_0$ and $t=r$, it follows from Part~(ii) of Theorem~\ref{thm:EKR} that   $ \mathcal{F}_n(K_{n_0}^{(r)})$ is a strong EKR family.  Thus, we conclude that  the  size of a intersecting family $\F$ among the elements of $\mathcal{F}_n(K_{n_0}^{(r)})$ is at most  $\binom{n-r}{n_0-r}$ and
  equality holds if and only if $\F$  consists of  all copies of $K_{n_0}^{(r)}$ in   $K_{n}^{(r)}$ that contain
a given subset of $r$ fixed vertices  of the vertex set of  $K_{n}^{(r)}$.
  

Now, we show that the triple $\big(\F_n(H), \mathcal{F}_n(K_{n_0}^{(r)}), \{\subseteq\}\big)$ satisfies the second condition of  Part~(1) in the definition of an EKR chains. 
To show this, we must prove that for every $Q\in\mathcal{F}_n(K_{n_0}^{(r)})$ the family $\LL_Q$ which is equal to the family of all subhypergraphs of $Q(= K_{n_0}^{(r)})$
 that are isomorphic to $H$, is EKR. This claim follows from the fact that $\F_{n_0}(H)$ is  $(S_{n_0},1)$-balanced, where $S_{n_0}$ is the 
 permutation group of the vertex of $K_{n_0}^{(r)}$.
Note that the action of the group $S_{n_0}$ on the set of the vertices naturally extends to the set of the hyperedges, and in turn, to the set of the copies of $H$ in $K_{n_0}^{(r)}$,
 i.e. $\F_{n_0}(H)$.

Since $K_{n_0}^{(r)}$ can be decomposed into disjoint copies of $H$,  Conditions~(ii) and (iii) of  Definition~\ref{balance} for being  $(S_{n_0},1)$-balanced hold.
 Therefore, $\LL_Q$ is in fact $(S_{n_0},1)$-balanced. Now, Lemma \ref{lem:gbalanceisEKR} guarantees that $\LL_Q$ is an  EKR family. 
 
Similar to the previous proof,  by symmetry each copy of $H$ is contained in the same number of copies of $K_{n_0}^{(r)}$ in $K_{n}^{(r)}$ 
and  each copy of $K_{n_0}^{(r)}$ in $K_n^{r}$ 
 contains the same number of copies of $H$.
  
Let  $Q_1 \in \mathcal{F}_n(K_{n_0}^{(r)})$. Take a $n_0$-subset $A$ in $V(K_{n}^{(r)})$ such that 
$|A\cap V(Q_1)|=r$. Consider the  complete $r$-uniform hypergraph  $Q_2$ with vertex set $A$. It follows directly that $|E(Q_1)\cap E(Q_2)|=1$.

  Consequently, it  follows by 
 the composition lemma that $\big(\mathcal{F}_n(H), \mathcal{F}_n(K_{n_0}^{(r)}),\{\subseteq\}\big)$ is a special EKR chain.

\end{proof}
 \section{Deferred Proofs}~\label{sec:Defer}

In this section, we present the proofs of Lemma \ref{lem:composition} and Lemma \ref{lem:gbalanceisEKR}.
Consider an EKR chain $(\mathcal{L},\mathcal{M},\sim_{\I})$.
  Since  for every $M\in \mathcal{M}$ and $i\in \mathcal{I}$, the family  $\mathcal{L}^{^{(i)}}_{_M}$ is an EKR family,
 for every $a\in M$, the family $\{L\in\mathcal{L}^{^{(i)}}_{_M}|\, a\in L \}$ has the largest possible size among all intersecting subfamilies in $\mathcal{L}^{^{(i)}}_{_M}$. 
Therefore, this 
 size is independent of the choice of $a$. This observation helps us to find the size of the largest intersecting subfamilies in 
 $\mathcal{L}^{^{(i)}}_{_M}$. The next lemma helps us to find this value.

\begin{lemma}\label{lem:size}
Let $\ell\leq m\leq n$ be positive integers. 
Let $\mathcal{L}$ and $\mathcal{M}$ be  families of $\ell$-subsets and $m$-subsets of an $n$-element set $X$, respectively. Assume that $\sim_{\I}$ 
 is a family of  regular relations  from $\LL$ to $\M$.
 \begin{itemize}
     \item[(i)]     If
     each element $a\in X$ belongs to the same number of the elements of $\mathcal{M}$,
       then for any $a\in X$, we have $|\M_a|=\tfrac{m}{n}|\mathcal{M}|$, where $\M_a\isdef\{M\in\M|\, a\in M\}$.
    \item[(ii)]     Let  $\mathcal{M}$ be an EKR family in $X$. Then, the size of a largest intersecting subfamily of $\mathcal{M}$ is $\tfrac{m}{n}|\mathcal{M}|$.
     \item[(iii)]  Assume that  $(\mathcal{L},\mathcal{M},\sim_{\I})$ is an EKR chain. Then,
      for every $M\in \mathcal{M}$ and $i\in\I$, the size of a largest intersecting subfamily of $\mathcal{L}^{^{(i)}}_{_M}$ is $\tfrac{\ell}{m}|\mathcal{L}^{^{(i)}}_{_M}|$.
      \item[(iv)]   Assume that  $(\mathcal{L},\mathcal{M},\sim_{\I})$ is an EKR chain. Then,
       every element $a\in X$ belongs to the same number of the elements of $\mathcal{L}$.

 \end{itemize}
\end{lemma}

\begin{proof}
 First, we  present a proof of Part~(i). 
Let $\mathcal{M}=\{M_1,M_2,\ldots,M_t\}$.
We count the number of $(a, M_j)$ where $a\in M_j$ in two ways. First, notice that each $M_j$ has size $m$, hence this number is equal to $mt$.
 Also, since every element $a\in X$ lies in $|\M_a|$ of the set $M_j$, we have   $n|\M_a|=mt$, and consequently, $|\M_a|={\tfrac{mt}{n}}$. 
 
   To prove (ii), since $\M$  is an  EKR family, for  any two elements $a_1,a_2 \in X$, we have
 $|\mathcal{M}_{a_1}|=|\mathcal{M}_{a_2}|$. Therefore, Part~(i) directly implies Part~(ii).

To prove (iii), since $\LL_{_{M}}^{^{(i)}}$  is an  EKR family, for  any two elements $b_1,b_2 \in M$, we have
$|{(\LL_{_{M}}^{^{(i)}})}_{b_1}|=|{(\LL_{_{M}}^{^{(i)}})}_{b_2}|$ where ${(\LL_{_{M}}^{^{(i)}})}_{b}\isdef\{L\in \LL_{_{M}}^{^{(i)}}|\, b\in L\}$ for any $b\in M$. 
Therefore, Part~(i) directly implies Part~(iii).

 To prove (iv), take an arbitrary element $a\in X$ and count the number of triples $(L,M,i)$ where $a\in L$, $L\in\LL$, $M\in\M$, and $L\sim_i M$ in two ways. We call each such triple, \textit{good}. 
 First, notice that
 the number of $\ell$-sets $L$ in $\LL$  which contain $a$ is equal to $|\LL_a|$. For any subset $L$, the number of good triples with the first part being $L$ is equal to 
  $\sum\limits_{j\in \I}{|\M^{^{(j)}}_L|}$. 
 According to Part~(iv) of the definition of the EKR chain, this number is independent from $L$. Therefore, the total number of good triples is 
 \[
 |\LL_a| \cdot\sum\limits_{j\in \I}{|\M^{^{(j)}}_L|}.
 \]

 For the second way of counting, note that $a$ appears in $|\M_a|$ $m$-subsets $M$ in $\M$ and for every choice of $i$, the subset $M$ contains exactly $|(\LL^{^{(i)}}_{_{M}})_a|$ $\ell$-subsets $L$ in $\LL$ such that $L\sim_i M$  and $a\in L$.  Therefore, the number of good triples is equal to 

  \[
  |(\LL^{^{(1)}}_{_{M}})_a|\cdot |\M_a|\cdot |\I|.
  \]
  In the above equation, we use the fact that $|(\LL^{^{(1)}}_{_{M}})_a|=|(\LL^{^{(i)}}_{_{M}})_a|$ for every $i\in\I$.
  Thus,
  $$|\LL_a| \cdot\sum\limits_{j\in \I}{|\M^{^{(j)}}_L|}=|(\LL^{^{(1)}}_{_{M}})_a|\cdot |\M_a|\cdot |\I|.$$
  From Parts~(ii) and (iii),  we have  $|\mathcal{M}_a|=\tfrac{m}{n}|\mathcal{M}|$ and 
 $|(\LL^{^{(1)}}_{_{M}})_a|={\tfrac{\ell}{m}}|(\LL^{^{(1)}}_{_{M}})|$. Therefore,  $|\mathcal{M}_a|$ and $|(\LL^{^{(1)}}_{_{M}})_a|$ do not depend on the choice of $a$. Then, we have
 \begin{align*}
 |\LL_a|&=\tfrac{|(\LL^{^{(1)}}_{_{M}})_a|\cdot |\M_a|\cdot |\I|}{\sum\limits_{j\in \I}|{\M^{^{(j)}}_{_L}}|}\\
 &=\tfrac{({\ell\over m}|\LL^{^{(1)}}_{_{M}}|)({m\over n}|\M|)\cdot |\I|}{\sum\limits_{j\in \I}|\M^{^{(j)}}_{_L}|}\\
 &= \tfrac{\ell|\LL^{^{(1)}}_{_{M}}|\cdot |\M|\cdot |\I|}{n(\sum\limits_{j\in \I}|\M^{^{(j)}}_{_L}|)}.
  \end{align*}
  
\end{proof}

 We are now ready to prove Lemma \ref{lem:composition}.

\begin{proof}[Proof of Lemma \ref{lem:composition}]
We first prove Part~(i) of the lemma.
Suppose that $(\mathcal{L},\mathcal{M},\sim_{\I})$ is an EKR chain. Consider a maximum size intersecting subfamily $\mathcal{L}'\subseteq \mathcal{L}$.
Define
\[
\widetilde{\M}\isdef\{ (M,\sim_i): M \in \mathcal{M} \text{ and } \exists L\in \mathcal{L}' \text{ such that }L \sim_i M \}.
\]

  Let $\mathcal{G}$ be a bipartite graph with parts $\mathcal{L}'$ and $\widetilde{\mathcal{M}}$.
The vertex $L\in \mathcal{L}'$ is adjacent to the vertex $(M,\sim_i)\in \widetilde{\mathcal{M}}$ when $L\sim_i M$. 
Define
 $$\M'\isdef\{M\in \M: \exists i\in \I \text { such that } (M,\sim_i) \in \widetilde{\mathcal{M}} \}.$$
  Since for each $M\in \M'$, there exist at most $|\I|$ elements of the form $(M,\sim_i)$ in $\widetilde{\mathcal{M}}$, we have 
\[
|\widetilde{\mathcal{M}}| \leq |\M'|\cdot |\I|.
\]

Note that  $\mathcal{M}'$ is an intersecting subfamily of $\mathcal{M}$. Indeed,
for every pair $M_1,M_2\in \mathcal{M}'$,  by the definition of $\mathcal{M}'$, there exist $i_1$ and $i_2$ in $\mathcal{I}$ such that
$(M_1,\sim_{i_1}) \in \widetilde{\mathcal{M}}$ and  
$(M_2,\sim_{i_2}) \in \widetilde{\mathcal{M}}$. By the definition of $\widetilde{\mathcal{M}}$, there exist $L_1$ and $L_2$ in $\mathcal{L}'$ such that 
$L_1\sim_{i_1} M_1$ and $L_2\sim_{i_2} M_2$.  In particular, $L_1\subseteq M_1$ and $L_2\subseteq M_2$.
Since $\mathcal{L}'$ is an intersecting subfamily of $\mathcal{L}$, therefore
we have $L_1\cap L_2\neq \varnothing$, and consequently,   $M_1\cap M_2  \neq \varnothing$.

Since $\M$ is an EKR family in $X$  and $\mathcal{M}'$ is an intersecting subfamily of $\mathcal{M}$, thus we may apply Part~(ii) of Lemma \ref{lem:size} to conclude that $|\mathcal{M}'|\leq \tfrac{m}{n}|\mathcal{M}|$. Therefore, $|\widetilde{\mathcal{M}}| \leq  \tfrac{m}{n}|\M| \cdot|\I|$.
\medskip

We now determine upper and lower bounds on the number of the edges of $\mathcal{G}$. For every vertex $(M,\sim_i) \in \widetilde{\mathcal{M}}$, consider the set of its neighbors in $\mathcal{L}'$. These neighbors correspond to an intersecting subfamily of $\mathcal{L}^{^{(i)}}_{_M}$. Since $\mathcal{L}^{^{(i)}}_{_M}$ is EKR family, thus  $(M,\sim_i)$ has at most $\tfrac{\ell}{m}|\mathcal{L}^{^{(i)}}_{_M}|$ neighbors in $\mathcal{L}'$, according to Part~(iii) of Lemma \ref{lem:size}. 

Since $\widetilde{\mathcal{M}}$ contains at most $\tfrac{m}{n}|\mathcal{M}|\cdot|\I|$ vertices and each vertex has at most $\tfrac{\ell}{m}|\mathcal{L}^{^{(i)}}_{_M}|$ neighbors in 
$\mathcal{L}'$, the number of the edges of $\mathcal{G}$ is at most $\tfrac{\ell}{n}|\mathcal{M}|\cdot|\mathcal{L}^{^{(i)}}_{_M}|\cdot |\I|$. 
 
 Take an arbitrary $L\in \mathcal{L}'$.  The degree of $L$ in $\mathcal{G}$ is equal to $\sum_{i\in \I}|\mathcal{M}^{^{(i)}}_{_L}|$. 
 According to Definition~\ref{def:EKRchain}, this sum is 
 independent of the choice of $L$.
Therefore, the total number of the edges of $\mathcal{G}$ is equal to $|\mathcal{L}'| \sum\limits_{i\in \I}|\mathcal{M}^{^{(i)}}_{_L}|$.
Comparing this to the upper bound on the number of edges of $\G$, namely $\tfrac{\ell}{n}|\mathcal{M}|\cdot |\I| \cdot |\mathcal{L}^{^{(i)}}_{_M}|$ we obtain the following inequality:
\begin{equation}\label{eq:boundsonedges}
     |\mathcal{L}'| (\sum\limits_{i\in \I}|\mathcal{M}^{^{(i)}}_{_L}|)\leq \tfrac{\ell}{n}|\mathcal{M}|\cdot|\mathcal{L}^{^{(i)}}_{_M}|\cdot |\I| 
\end{equation}

Consequently, we have

\begin{equation}\label{eq:LprimelessthanLa}
    |\LL'| \leq \tfrac{\ell |\mathcal{M}|\cdot|\mathcal{L}^{^{(i)}}_{_M}|\cdot |\I|}{n (\sum\limits_{i\in \I}|\mathcal{M}^{^{(i)}}_{_L}|)} = |\LL_a|.
\end{equation}
    Note that the last equality follows from Part~(iv) of Lemma~\ref{lem:size}



This concludes the assertion of Part~(i). Note that if $\LL'$ is a maximum intersecting subfamily of $\LL$, we must have $|\LL'|=|\LL_a|$. Furthermore, 
each inequality in the preceding proof becomes an equality. In particular, $\M'$ is a maximum intersecting subfamily of $\M$. 
This  observation is crucial for the proof of the next part of the lemma. 
\\

Now, suppose that $(\mathcal{L},\mathcal{M},\sim_{\I})$ is a special EKR chain. Let $\LL'$ be a maximum size intersecting subfamily of $\LL$. We must prove that there exists an element $a$ such that all the members of $\LL'$ include $a$. 
As we mentioned above, $\M'$ is a maximum intersecting subfamily  $\M$. Since $(\mathcal{L},\mathcal{M},\sim_{\I})$ is assumed to be a special EKR chain, the subfamily of $\M'$ is identical to a 
subfamily $\M_a$ for some element $a$. We prove that for this 
choice of $a$ we have $\LL'=\LL_a$. For a contradiction, suppose that there exists $L_1\in\LL'$ such that $a\notin L_1$. 
Since in Inequality \ref{eq:boundsonedges}, the equality holds and the right hand side is not $0$ therefore the left hand side is not $0$. Hence, $\sum\limits_{i\in \I}|\mathcal{M}^{^{(i)}}_{_L}|\neq 0$. Consequently,
there exists $M_1\in\M$  and ${i_1} \in \I$ such that $L_1\sim_{i_1} M_1$.
 As $L_1\subseteq M_1$, we have $M_1\in \M'=\M_a$. 
 
 Let $M_2$ be an 
element of $\M$ such that $M_1\cap M_2 =\{a\}$. 
Notice that such $M_2$ exists as $(\mathcal{L},\mathcal{M},\sim_{\I})$ is a special EKR chain. Therefore, $M_2\in \M_a= \M'$.
 Since $M_2\in \M'$  there exists at least one element $L_2\in \LL'$ and $i_2\in \I$ such that $L_2\sim_{i_2} M_2$, and consequently, $L_2\subseteq M_2$. 
Therefore, $L_1,L_2$ belong to the 
intersecting subfamily $\LL'$. This is a contradiction since $L_1\cap L_2 \subseteq M_1\cap M_2=\{a\}$, while $a\notin L_1$ and $L_1\cap L_2\neq \emptyset$.

\end{proof}




\begin{proof}[Proof of Lemma \ref{lem:gbalanceisEKR}]

We  apply Part~(i) of the composition lemma.  In this lemma, take   $\LL=\F$, $\M=\{X\}$, and $\I=  G$. For every $g\in G$, define $\sim_g$ to be the following regular relation induced by $g$.
For any element $F\in \F$, define $F\sim_g X$ if and only if $gF\in \{D_1,D_2,\ldots,D_r\}$. Note that $\sim_g$ is a regular relation since $F\subseteq X$ regardless of any extra condition. 

We claim that $(\F,\{X\},\sim_G)$ is an EKR chain. The condition that $\M$ is an EKR family holds trivially since $\M$ consists of only one element, namely $X$.

Note that $\LL^{(g)}_X = \{g^{-1}D_1, g^{-1}D_2, \ldots, g^{-1}D_r\}$. The second requirement for being an EKR chain is to show that, for every \(g \in G\), 
the set $\LL^{(g)}_X$ forms an EKR family. This condition is satisfied due to Condition~(iii) of the definition of a $(G, j)$-balanced family.
Since for any two members  $g, g'$ of $G$ we have $|\LL^{(g)}_X| =|\LL^{(g')}_X |=r$, the third condition of being an EKR chain holds.
Take two arbitrary elements $F, F'\in \mathcal{F}$. There exists some $g\in G$ such that $gF=F'$. 
Let $g_1,\ldots,g_r$ be elements of  $G$ such that $F\sim_{g_i} X$ for each $1\leq i\leq r$. One can see that 
$g_1g^{-1},\ldots,g_rg^{-1}$ are elements of  $G$ such that $F'\sim_{g_ig^{-1}} X$ for each $1\leq i\leq r$. Therefore, the fourth condition of the definition of the EKR chain  also holds.
Thus, the assertion of the lemma follows directly from Part~(i) of Lemma \ref{lem:composition}.
\end{proof}


\begin{proof}[Proof of Claim \ref{claim}]

 Let  $\F$ be a maximum intersecting family of $\mathcal{B}_n(K_{k,k})$. Define $\F_X$ and $\F_Y$ as follows,
$$\F_X\isdef\{A|\, \exists Q\in  \mathcal{B}_n(K_{k,k}) \text{ such that }  A=V(Q)\cap X\}$$ and 
$$\F_Y\isdef\{B|\, \exists Q\in \mathcal{B}_n(K_{k,k}) \text{ such that }  B=V(Q)\cap Y\}.$$  
\medskip 

Since    $\F$ is a maximum intersecting family of $\mathcal{B}_n(K_{k,k})$,
  one  can check that $\F_X$ and $\F_Y$ must be maximum intersecting families of $k$-subsets in $X$ and $Y$, respectively.
   Since $n\geq 2k$, by  Theorem~\ref{thm:EKR} we have $|\F_X|= {n-1\choose k-1}$ and $|\F_Y|= {n-1\choose k-1}$, and consequently,
   $|\F|= {n-1\choose k-1}^2.$  For any $xy\in E(K_{n,n})$, the cardinality of $\{Q\in \mathcal{B}_n(K_{k,k})|\, xy\in E(Q)\}$
   is equal to ${n-1\choose k-1}^2$.   Then, $\mathcal{B}_n(K_{k,k})$ is an  EKR family.

 Now assume that $n> 2k$. By  Theorem~\ref{thm:EKR},  there exist $x\in X$ and $y\in Y$ such that
  $$\F_X=\{A|\, x\in A, |A|=k, \text{and } A\subset X\}$$ and
  $$\F_Y=\{B|\, y\in B, |B|=k, \text{and } B\subset Y\}.$$ 
  Therefore,
    $$\F=\{Q\in \mathcal{B}_n(K_{k,k})|\, xy\in E(Q)\},$$
    that is, all members of $\F$ contains the edge $xy$.   Hence,  $\mathcal{B}_n(K_{k,k})$ is a strong  EKR family.
\end{proof}
\section*{Acknowledgment}
This work is based upon research funded by Iran National Science Foundation
(INSF) under project No.4030010. The authors are deeply grateful to Amir Daneshgar for his helpful comments.


\end{document}